\documentclass[conference]{IEEEtran}
\IEEEoverridecommandlockouts

\usepackage{url}
\usepackage{hyperref}
\usepackage{amsmath,amsfonts,amssymb,graphicx,subfigure,color}
\usepackage{amsthm}
\usepackage[british]{babel}
\usepackage{hhline}
\usepackage{multirow}
\usepackage{array}
\usepackage{booktabs}
\usepackage[footnotesize]{caption}

\parskip = \medskipamount

\renewcommand{\theenumi}{(\roman{enumi})}

\newcommand{\real}{\ensuremath{\mathbb{R}}}

\newcommand{\Ac}{\mathcal{A}}

\newcommand{\Fc}{\mathcal{F}}

\newcommand{\Sc}{\mathcal{S}}
\newcommand{\Tc}{\mathcal{T}}

\newcommand{\Xc}{\mathcal{X}}

\theoremstyle{plain}

\newtheorem{theorem}{Theorem}[section]

\newtheorem{lemma}[theorem]{Lemma}

\newtheorem{definition}[theorem]{Definition}

\newcommand{\longthmtitle}[1]{\mbox{}{\bf \textit{(#1).}}}

\newcommand{\ones}{\mathbf{1}}

\DeclareMathOperator*{\argmin}{argmin}

\begin{document}

\title{Zero-Knowledge Proof-Based Approach for Verifying the Computational Integrity of Power Grid Controls
\thanks{This work was authored in part by the National Renewable Energy Laboratory (NREL), operated by Alliance for Sustainable Energy, LLC, for the U.S. Department of Energy (DOE) under Contract No. DE-AC36-08GO28308. This work was supported by the Laboratory Directed Research and Development (LDRD) Program at NREL. The views expressed in the article do not necessarily represent the views of the DOE or the U.S. Government. The U.S. Government retains and the publisher, by accepting the article for publication, acknowledges that the U.S. Government retains a nonexclusive, paid-up, irrevocable, worldwide license to publish or reproduce the published form of this work, or allow others to do so, for U.S. Government purposes.}
}

\author{\IEEEauthorblockN{ Chin-Yao Chang, \quad Richard Macwan, \quad Sinnott Murphy}
\IEEEauthorblockA{
\textit{National Renewable Energy Laboratory}, Golden, USA\\
chinyao.chang@nrel.gov, richard.macwan@nrel.gov, sinnott.murphy@nrel.gov}
}
\maketitle
\begin{abstract}
    The control of future power grids is migrating from a centralized to a distributed/decentralized scheme to enable a massive penetration of distributed energy resources and bring extreme enhancements of autonomous operations in terms of grid resilience, security, and reliability. Most effort has been on the design of distributed/decentralized controllers; however, the guarantees of the proper execution of the controls are also essential but relatively less emphasized. A common assumption is that local controllers would fully follow the designated controller dynamics based on the data received from communication channels. Such an assumption could be risky because proper execution of the controller dynamics is then built on trust in secure communication and computation. On the other hand, it is impractical for a verifier to repeat all the computations involved in the controls to verify the computational integrity. In this work, we leverage a type of cryptography technology, known as zero-knowledge scalable transparent arguments of knowledge to verify the computational integrity of control algorithms, such that verifiers can check the computational integrity with much less computational burden. The method presented here converts the challenge of data integrity into a subset of computational integrity. In this proof-of-concept paper, our focus will be on projected linear dynamics that are commonly seen in distributed/decentralized power system controllers. In particular, we have derived polynomial conditions in the context of zk-STARKs for the projected linear dynamics. 
\end{abstract}

\begin{IEEEkeywords}
Cybersecurity, cryptography, gradient control algorithms, power system control, zk-STARK
\end{IEEEkeywords}

\section{Introduction}

As the energy grid of today is rapidly evolving into a complex and highly distributed cyber-physical system, as illustrated in Figure~\ref{Control_Levels}, new control algorithms and approaches are being developed for the reliable and efficient operation of the grid. These control algorithms have also played a huge role in the rapid integration of Distributed Energy Resources (DERs), Electric Vehicles (EVs), and smart loads into the grid. Control algorithms such as \cite{zhang2013aggregated,koch2011modeling,utkarsh2020model} provide grid services using behind-the-meter DERs and home smart loads, \cite{WALD2022118753,yan2018optimized} leverage smart building controls and EV charging infrastructure to support reliable grid operations. Whether these control approaches span from local home-or building-level control to the substation-or system-level control, most of these approaches are highly dependent on the cyber or the communication systems for the exchange of measurements and control commands to control the physical energy system. Hence, these control algorithms form a key layer of interaction between the cyber and the physical layers of the energy system. An attack on the cyber layer-such as malicious data injection, or denial-of-service attacks that cause data drops or delays can have a huge impact on the accurate execution at the control layer and can be translated into large-scale physical system disruption. To protect the system against such threats, the traditional cybersecurity mechanisms for energy systems have primarily focused on either developing better encryption mechanisms for supervisory control and data acquisition (SCADA) protocols such as secure Distributed Network Protocol 3 (DNP3)~\cite{majdalawieh2007dnpsec,amoah2016securing} and secure Modbus~\cite{fovino2009design}, or on developing novel intrusion detection techniques for the communication network~\cite{yang2014multiattribute}; however, unlike these traditional cybersecurity mechanisms, this work focuses on injecting security into the control layer of the grid by proposing a verification framework using zero-knowledge proof. 

\begin{figure*}[h]
\centerline{\includegraphics[width=1.0\linewidth]{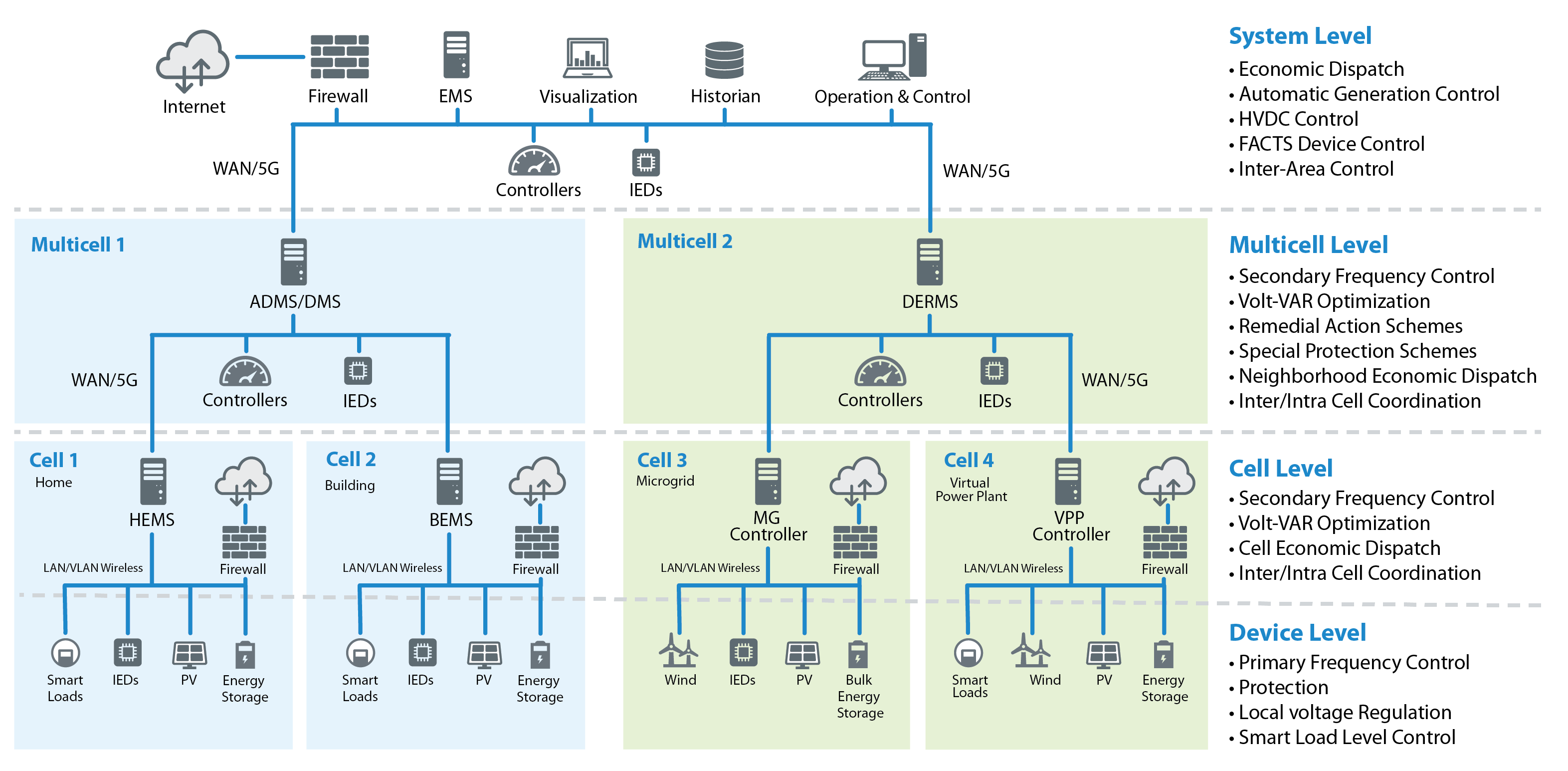}}
\caption{Energy cyber-physical system architecture and controls at various levels }
\label{Control_Levels}
\end{figure*}

A zero knowledge proof (ZKP) is a method by which one sender device (the prover) can prove to a receiver device (the verifier) that a given statement is true without conveying any additional information about the underlying statement. In the context of power grid control, a receiving device for a voltage set point value, for example, resulting from a control algorithm computation performed on a sender device, will also receive a zero-knowledge proof of the correct computation with respect to the set point value. Because no knowledge about the underlying computation is being transmitted as part of the proof it is impossible to recreate a proof of computation by a malicious device on the network. Hence, by focusing on verifying the computational integrity for control algorithms, this novel concept converts the challenge of traditional data integrity into a subproblem under computational integrity and helps both prevent and detect attacks at the control layer. Although there are many ZKP approaches, such as zero-knowledge succinct noninteractive arguments for knowledge (zk-SNARK)~\cite{6956581} and bulletproofs~\cite{8418611}, in this work, we leverage the zero knowledge succinct transparent arguments of knowledge (zk-STARK) scheme for verifying the computational integrity of the power grid control algorithm. Compared to other ZKP schemes, zk-STARK is scalable, fast, requires no trusted setup, and is plausibly post-quantum secure and hence is chosen for this study. A detailed comparison of different ZKP schemes can be found in~\cite{9758531,ZKSTARK-SNARK-1,ZKSTARK-SNARK-2}. We envision that zk-STARK might help address the cybersecurity challenges described. Leveraging zk-STARK to a general controller dynamics is probably too ambitious for the first step; therefore, we focus on projected linear controller dynamics, which are commonly seen not only in grid control~\cite{nahata2017decentralized,chang2019saddle,jiao2020decentralized} but also robotic~\cite{bullo2009distributed}; and wireless networks~\cite{ram2009distributed}. 

\textit{Contributions:} 
In this work, we leverage zk-STARK methodologies to verify the accurate execution of projected linear dynamics that are commonly seen in gradient-based control algorithms. The current form of zk-STARK cannot directly fit with the control dynamics mainly because the projection operators involved in the controller cannot be written in polynomial constraints. We address the issue by proposing a two-stage verification framework. The framework involves some reformulation of the original projected linear dynamics for convenience of separating the verification process into the online and offline stages. The online stage involved only some logical checks that can be done in real time for the verifier. The offline stage involves verification on more complicated computations, which is where zk-STARK comes into play. The numerical studies provide a good proof of concept of the proposed framework.  
 
The reminder of the paper is organized as follows. Section~\ref{sec:prelim} goes through some preliminary concepts. Section~\ref{sec:veri} shows how a class of projected linear dynamics is reformulated into polynomial constraints that fit in zk-STARK framework. Section~\ref{sec:num} uses a simple numerical example for the purpose of proof of concept.

\section{Preliminary}\label{sec:prelim}
\textit{Notations}: Let $\mathbb{R}$ be a set of real numbers; the superscript of $\mathbb{R}$ (if exists) denotes the dimension. For a scalar,  $x\in\mathbb{R}$, $\lfloor x \rfloor$ is the flooring value of $x$. For a vector, $x\in\mathbb{R}^n$, $x_i$ is referred as the $i^{th}$ element of $x$. Given a closed set, $\Xc$, we define a projection operator as $\text{Proj}_{\Xc}(x_0):= \argmin_{x} \|x - x_0\|$. The degree of a polynomial $f$ is denoted as $\text{deg}(f)$. A finite (Galois) field with a prime power, $q^k$, is denoted as $\Fc_{q^k}$, where $q$ is a prime number, and $k$ is a positive integer. $\Fc_{q^k}^*:= \Fc_{q^k}\setminus \{0\}$ is the multiplicative group of $\Fc_{q^k}$. 

\subsection{Finite fields}
Here we list some concepts of finite fields that are very relevant to zk-STARK. Comprehensive materials for finite fields can be found in~\cite{grove2012algebra}, and a more compact overview of the relevant concepts is available in~\cite{IntroFiniteField}.
\begin{definition}\longthmtitle{Prime fields}
For every prime number, $q$, the set $\Fc_{q} = \{0, 1, \cdots, q-1\}$ forms a field under mod$-q$ addition and multiplication.
\end{definition}

\begin{theorem}\longthmtitle{Cyclic groups~\cite{IntroFiniteField}}
The elements of a cyclic group $G$ of order $n$ with generator $g$ are $\{g^0,g^1,\cdots,g^{n-1}\}$. The multiplication rule is $g^i * g^j = g^{(i+j) \text{mod } n}$, the identity is $g^0=1$, and the inverse of $g^i \not=1$ is $g^{n-i}$.
\end{theorem}

\begin{theorem}\longthmtitle{Cyclic subgroup~\cite{grove2012algebra}}
If $G$ is a finite subgroup of the multiplicative group, $\Fc_{q^k}^*$, of a field $\Fc_{q^k}$, then $G$ is cyclic. 
\end{theorem}

\subsection{An overview of zk-STARK}
Figure \ref{zkSTARK} illustrates the steps involved in the zk-STARK framework and is described in detail in \cite{DBLP:journals/iacr/Ben-SassonBHR18}. In the context of power grid control computational verification, the first step is to identify the aspect of control algorithm computation that needs to be verified and can be understood as the computational integrity statement for the purposes of zk-STARK implementation. Similar to other ZKPs, zk-STARK framework aims to verify the computational integrity statement by translating the statement into a polynomial over a finite field. The process of translating a computational integrity statement along with an execution trace for the computation into a low-degree polynomial is referred to as algebraic intermediary representation (AIR). The process of AIR can encompass either a single computation of multiple computations that can be translated into a single low-degree polynomial over a finite field. Finally, during the low complexity verification part of the zk-STARK framework, the verifier aims to verify that the polynomial representing the computation or a set of computations is of sufficiently low degree using random queries to the prover. If the verifier can be convinced through a limited set of random queries to the prover of the low degree of the polynomial over a finite field, then the computation is considered verified. 

\begin{figure}[h]
\centerline{\includegraphics[width=1.0\linewidth]{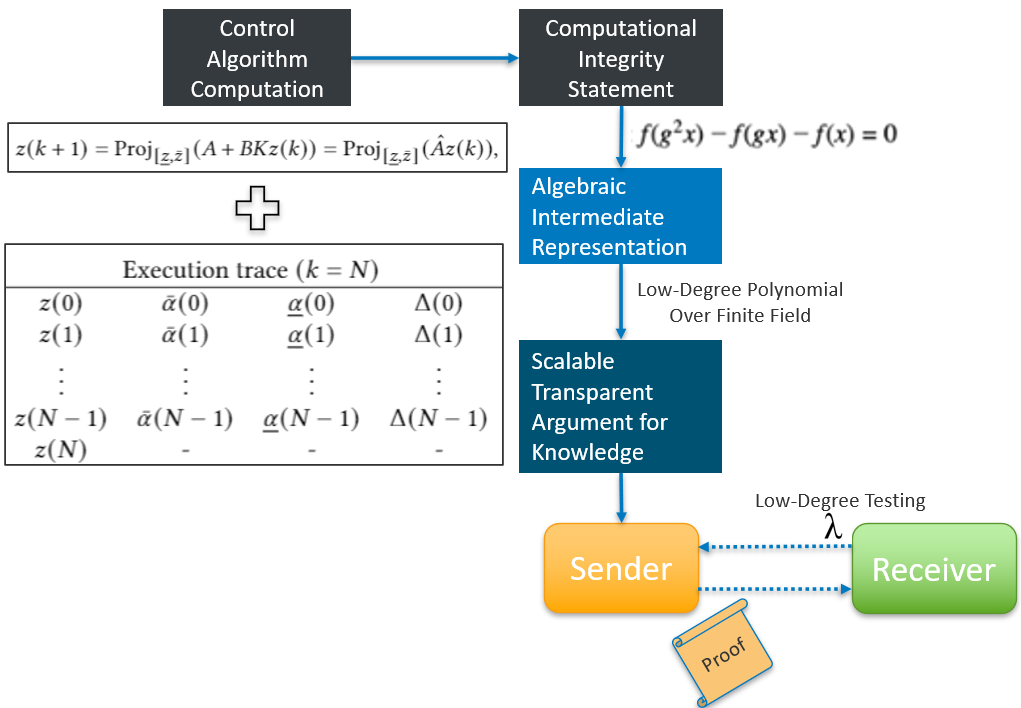}}
\caption{The zk-STARK framework }
\label{zkSTARK}
\end{figure}

\section{Verification of Algorithm Implementation}\label{sec:veri}

We consider a scenario that a controller should obligate the following linear dynamics in controlling system state variable $z$:
\begin{align}\label{eq:control}
    z(k+1) = \text{Proj}_{[\underline{z},\bar{z}]}(A+BK) z(k) = \text{Proj}_{[\underline{z},\bar{z}]}(\hat{A} z(k)),
\end{align}
where $z\in\real^n$, $A \in \real^{n\times n}$, $B\in\real^{n\times m}$, and $K\in\real^{m\times n}$. This state feedback control is commonly seen in various applications. There could be other variants of~\eqref{eq:control} that use the output feedback or the adaptive feedback gain $K$. Because those are not the focus of this work, we assume that the equivalent state-transition matrix, $\hat{A}$, is given in the remainder of the work. The goal of the verifier is to check whether~\eqref{eq:control} is properly executed for the number of steps of interest, $k\in\{0,\cdots,N-1\}$, without repeating the computations in~\eqref{eq:control}. In the context of zk-STARKs, the execution trace is a table that lists all the variables from the beginning to the end, $z(0),\cdots,z(N-1)$ for~\eqref{eq:control}. The polynomial constraints associated with the execution trace are polynomial equations that hold if and only if the designated method of generating the execution trace, equation~\eqref{eq:control}, is correctly executed through every time step from step $0$ to $N-1$. If the polynomial constraints have recurring patterns, then a low computational complexity verification method can be derived. In fact,~\eqref{eq:control} by itself already provides a recurring pattern that holds for all $k\in\{0,\cdots,N-1\}$. The issue is that fundamentally,~\eqref{eq:control} cannot be equivalently rewritten as polynomial equations because of the projection operator. We will address the issue by introducing some slack variables and by developing a two-stage verification framework. Intuitively, the verification of accurate projection operation is done in the first stage (online), and the verification of the remaining computations is done in the second stage (offline) where polynomial constraints are very useful.

\subsection{Equivalent formulations of the projected linear dynamics}
In this section, we show some equivalent formulations of~\eqref{eq:control} that lay the foundation for the two-stage verification framework. We first introduce the slack variables $\bar{\alpha}$ and $\underline{\alpha}$ for all $i\in\{1,\cdots,n\}$, with the meaning of the projection decisions as shown in~\eqref{eq:control_slack}:
\begingroup
\allowdisplaybreaks
\begin{subequations}\label{eq:control_slack}
\begin{align}
    z_i(k+1) &= \bar{\alpha}_i(k)\underline{\alpha}_i(k) \Big(\hat{A}_i z(k) \Big)+ (1 - \bar{\alpha}_i(k) )\bar{z}_i(k)  \label{eq:control_slack-1} \\ \nonumber 
    &+ (1 - \underline{\alpha}_i(k) )\underline{z}_i(k), \quad\quad \forall i=1,\cdots,n  \\
    \Delta_i(k) &= \bar{\alpha}_i(k) \cdot (\bar{z}_i - \hat{A}_i z(k) ) + \underline{\alpha}_i(k) \cdot (\hat{A}_i z(k) - \underline{z}_i), \label{eq:control_slack-2} \\ \nonumber & \hspace{5mm}\forall i=1,\cdots,n  \\
    \bar{\alpha}_i(k) &= \begin{cases}
        1 & \text{ if } \hat{A}_i z(k) \leq \bar{z}_i \\ 0 & \text{ if } \hat{A}_i z(k) > \bar{z}_i
    \end{cases}, \quad\quad \forall i=1,\cdots,n \label{eq:control_slack-3} \\
    \underline{\alpha}_i(k) &= \begin{cases}
        1 & \text{ if }\hat{A}_i  z(k) \geq \underline{z}_i \\ 0 & \text{ if } \hat{A}_i z(k) < \underline{z}_i
    \end{cases}, \quad\quad \forall i=1,\cdots,n \label{eq:control_slack-4} 
\end{align}
\end{subequations}
\endgroup
where $\hat{A}_i$ is the $i^{th}$ row of $\hat{A}$. The variable $\bar{\alpha}_i$ is defined such that it takes the value $1$ if the projection is inactive and $0$ otherwise; $\underline{\alpha}_i$ is defined similarly. For convenience of notation, we write $\bar{\alpha}$, $\underline{\alpha}$ and $\Delta$, respectively, as the concatenation of all the associated variables for all $i = 1,\cdots,n$.  It is straightforward to see that starting from the same $z(0) = z_{ini}$,~the sequence of $z$ generated from \eqref{eq:control} and \eqref{eq:control_slack} are the same. Note that variable $\Delta$ does not affect the dynamics of $z$, $\bar{\alpha}$, and $\underline{\alpha}$. In other words, dropping~\eqref{eq:control_slack-2} does not change the equivalency between~\eqref{eq:control} and~\eqref{eq:control_slack} in the sense of the overlapped trajectory of $z$. Introducing $\Delta$ paves the road for a further reformulation shown in Lemma~\ref{lem:equi_alpha}.

\begin{lemma}\label{lem:equi_alpha}
\longthmtitle{Equivalent formulation of~\eqref{eq:control_slack}}
Equation~\eqref{eq:control_slack} holds if and only if~\eqref{eq:control_slack-1},~\eqref{eq:control_slack-2},  and~\eqref{eq:equi_alpha} hold:
\begin{subequations}\label{eq:equi_alpha}
\begin{align}
    &\bar{\alpha}_i(k) \in\{0,1\}, \quad \underline{\alpha}_i(k) \in\{0,1\}, \quad \forall i = 1,\cdots,n, \label{eq:equi_alpha-1} \\
    & \Delta_i(k) \geq \bar{z}_i - \underline{z}_i, \quad \forall i = 1,\cdots,n,  \label{eq:equi_alpha-2} \\
    & z(k+1) \in [\underline{z},\bar{z}].  \label{eq:equi_alpha-3}
\end{align}
\end{subequations}
\end{lemma}
\begin{proof}We prove the lemma by case-by-case analyses:\vspace{-3mm}
\begin{itemize}
    \item Case $\underline{z}_i \leq \hat{A}_i z(k) \leq \bar{z}_i$: the only $\bar{\alpha}_i$ and $\underline{\alpha}_i$ such that~\eqref{eq:equi_alpha-2} holds is $\bar{\alpha}_i = \underline{\alpha}_i= 1$. The selection of $\bar{\alpha}_i$ and $\underline{\alpha}_i$ satisfies~\eqref{eq:equi_alpha-1} and \eqref{eq:equi_alpha-3} as well. The solution concurs with~\eqref{eq:control_slack-3} and~\eqref{eq:control_slack-4}.
    
    \item Case $\hat{A}_i z(k) > \bar{z}_i$:   $\underline{\alpha}_i$ must be $1$ to ensure~\eqref{eq:equi_alpha-1} and \eqref{eq:equi_alpha-2} hold; $\bar{\alpha}_i$ must be $0$,~\eqref{eq:equi_alpha-3} is violated otherwise. The solution $\underline{\alpha}_i=1$ and $\bar{\alpha}_i=0$ concurs with~\eqref{eq:control_slack-3} and~\eqref{eq:control_slack-4}.
    
    \item Case $\hat{A}_i z(k) < \underline{z}_i$: the arguments are similar to the one for the previous case. 
\end{itemize}
\vspace{-5mm}
\end{proof}
To this end, we have two equivalent formulations of~\eqref{eq:control}: The first reformulation is equation~\eqref{eq:control_slack} and the other is~\eqref{eq:control_slack-1}, \eqref{eq:control_slack-2}~and~\eqref{eq:equi_alpha}. The reason doing all these reformulations is that it is relatively easy to build a low-complexity verification method on the last reformulation. The prover/controller is required to compute more computationally heavy dynamics~\eqref{eq:control_slack} in place of~\eqref{eq:control}; and the verifier verifies the accurate execution of~\eqref{eq:control_slack} by checking~\eqref{eq:control_slack-1}, \eqref{eq:control_slack-2}, and~\eqref{eq:equi_alpha} without repeating most computations in~\eqref{eq:control_slack}.    
\subsection{Two-stage verification framework}

The two-stage verification framework has online and offline stages of computation and verification. In the online stage, for every iteration $k$, the controller/prover computes~\eqref{eq:control_slack} to generate the variables for the next time step. The verifier checks whether~\eqref{eq:equi_alpha-2} and~\eqref{eq:equi_alpha-3} are satisfied. Checking~\eqref{eq:equi_alpha-2} and~\eqref{eq:equi_alpha-3} should not pose a computational burden on the verifier even in the real-time setups. If~\eqref{eq:equi_alpha-2} and~\eqref{eq:equi_alpha-3} hold, then all $z(k+1)$, $\Delta(k)$, $\bar{\alpha}(k)$. $\underline{\alpha}(k)$ are accepted, and recorded on the execution trace. Otherwise, the results from the controller/prover are not accepted and the prover is requested to redo the computations for iteration $k$ until the results are valid. Table~\ref{tab:exe_trace} illustrates how the execution trace expands.
\begin{table*}[h]
    \centering
\begin{tabular}{ c  c c c | c c c c | c c c c }
 \multicolumn{4}{c|}{Execution trace $(k=0)$}  & \multicolumn{4}{c|}{Execution trace $(k=1)$} & \multicolumn{4}{c}{Execution trace $(k=N)$} \\  \hline
  $z(0)$ & - & - & -  & $z(0)$ & $\bar{\alpha}(0)$ & $\underline{\alpha}(0)$ & $\Delta(0)$ & $z(0)$ & $\bar{\alpha}(0)$ & $\underline{\alpha}(0)$ & $\Delta(0)$ \\
  - & - & - & -       & $z(1)$ & - & - & - & $z(1)$ & $\bar{\alpha}(1)$ & $\underline{\alpha}(1)$ & $\Delta(1)$ \\
  - & - & - & -       & - & - & - & - & \vdots & \vdots & \vdots & \vdots \\
  - & - & - & -       & - & - & - & - & $z(N-1)$ & $\bar{\alpha}(N-1)$ & $\underline{\alpha}(N-1)$ & $\Delta(N-1)$  \\
  - & - & - & -       & - & - & - & -  & $z(N)$ & - & - & - 
\end{tabular}
    \caption{An example of how the execution trace expands along with increasing numbers of iterations.}
    \label{tab:exe_trace}
\end{table*}

In the online stage, the verifier verifies only~\eqref{eq:equi_alpha-2} and~\eqref{eq:equi_alpha-3} and leaves the verification of~\eqref{eq:control_slack-1}, \eqref{eq:control_slack-2}, and~\eqref{eq:equi_alpha-1} to the offline stage. Note that we can also put the verification of~\eqref{eq:equi_alpha-1} online instead of offline because those involve only some simple computations. Here, we put it in the offline stage only because the low-complexity verification framework can incorporate the verification of~\eqref{eq:equi_alpha-1} as well. The offline verification assumes that every variable only takes values in a finite field and leverages some properties of finite fields. The assumption is technically not true given that $z$ and $\Delta$ are vectors taking real values; however, given that a real number point is stored in a fixed number of bits of computer memory, realistically, all the values of $z$ and $\Delta$ are still in a finite field $\Fc_{2^M}$, where $M$ is the number of bits for a floating number. The assumption is therefore valid, and we will assume that holds true in the rest of the paper. 

One main reason for a low-complexity verification on the verifier end is that the prover end provides some assists. When generating the execution trace, the prover is required to find polynomials such that:
\begingroup
\allowdisplaybreaks
\begin{subequations}\label{eq:poly_variables}
\begin{align}
    & f_{i,z}(g^k) = z_i(k), \quad \forall k\in\{0,\cdots,N\}, \quad \forall i \in 1,\cdots,n \\
    & f_{i,\bar{\alpha}}(g^k) =  \bar{\alpha}_i(k), \quad \forall k\in\{0,\cdots,N-1\}, \quad \forall i \in 1,\cdots,n \\
    & f_{i,\underline{\alpha}}(g^k) =  \underline{\alpha}_i(k),  \quad \forall k\in\{0,\cdots,N-1\}, \quad \forall i \in 1,\cdots,n  \\
    & f_{i,\Delta}(g^k) =  \Delta_i(k), \quad \forall k\in\{0,\cdots,N-1\}, \quad \forall i \in 1,\cdots,n
\end{align}
\end{subequations}
\endgroup
where $g$ is a generator of a cyclic subgroup, $\Sc$, of a multiplicative group of $\Fc_q$, $\Fc_q^*$, where $q$ is a prime power. The size of the cyclic subgroup is selected such that $|\Sc| = N+1$. All $f_{i,\bar{\alpha}}$, $ f_{i,\underline{\alpha}}$, and $f_{i,\Delta}$ are the polynomials with degree at most $N$; and  $f_{i,z}$ has the degree at most $N+1$. The degrees of those polynomials are denoted, respectively, as $D_{i,z}$, $D_{i,\bar{\alpha}}$, $ D_{i,\underline{\alpha}}$ and $D_{i,\Delta}$. Substituting the variables of~\eqref{eq:control_slack-1}, \eqref{eq:control_slack-2},  and~\eqref{eq:equi_alpha-1} by the polynomials~\eqref{eq:poly_variables} leads to some polynomial equations: $\forall k\in\{0,\cdots,N-1\}$ and $\forall i = 1,\cdots,n,$:
\begingroup
\allowdisplaybreaks
\begin{subequations}\label{eq:poly_offline}
\begin{align}
    &f_{i,z}(g^{k+1}) - f_{i,\underline{\alpha}}(g^k)f_{i,\bar{\alpha}}(g^k) \Big(\sum_{j=1}^n\hat{A}_{ij} f_{j,z}(g^k) \Big)  \label{eq:poly_offline-1} \\ \nonumber &\hspace{10mm} - (1 - f_{i,\bar{\alpha}}(g^k) )\bar{z}_i(k) - (1 - f_{i,\underline{\alpha}}(g^k) )\underline{z}_i(k) = 0, \\
    &f_{i,\Delta}(g^k) - f_{i,\bar{\alpha}}(g^k) \cdot \Big(\bar{z}_i - \sum_{j=1}^n\hat{A}_{ij} f_{j,z}(g^k) \Big) \label{eq:poly_offline-2} \\ \nonumber & \hspace{10mm}- f_{i,\underline{\alpha}}(g^k) \cdot \Big(\sum_{j=1}^n\hat{A}_{ij} f_{j,z}(g^k) - \underline{z}_i\Big) = 0,  \\
    &f_{i,\bar{\alpha}}(g^k)\Big(1-f_{i,\bar{\alpha}}(g^k)\Big) = 0, \label{eq:poly_offline-3} \\
    &f_{i,\underline{\alpha}}(g^k)\Big(1-f_{i,\underline{\alpha}}(g^k)\Big) = 0. \label{eq:poly_offline-4}
\end{align}
\end{subequations}
\endgroup
The polynomial constraints associated with~\eqref{eq:control_slack-1}, \eqref{eq:control_slack-2},  and~\eqref{eq:equi_alpha-1} for all $i = 1,\cdots,n$ are next given as:
\begin{subequations}\label{eq:poly}
\begin{align}
    &f_{i,z}(g^{0}) = z_{ini},
    \label{eq:poly-1} \\
    &\eqref{eq:poly_offline} \text{ holds}, \quad \forall k\in\{0,\cdots,N-1\}. \label{eq:poly-2}
\end{align}
\end{subequations}
Equation~\eqref{eq:poly-1} is about the initialization,  and~\eqref{eq:poly-2} is equivalent to~\eqref{eq:control_slack-1}, \eqref{eq:control_slack-2}, and~\eqref{eq:equi_alpha-1}. The verifier does not straight check whether~\eqref{eq:poly}  holds because the computations involved for the verification remain high. The verifier further requests the prover to substitute $g^k$ in~\eqref{eq:poly_offline} by a generic variable $x$, ($g^{k+1} = g\cdot x$) as shown in~\eqref{eq:poly_offline_p}, and construct the composition polynomials shown in~\eqref{eq:composite_poly}:
\begingroup
\allowdisplaybreaks
\begin{subequations}\label{eq:poly_offline_p}
\begin{align}
    &f_{i,z}(g x) - f_{i,\underline{\alpha}}(x)f_{i,\bar{\alpha}}(x) \Big(\sum_{j=1}^n\hat{A}_{ij} f_{j,z}(x) \Big)  \label{eq:poly_offline_p-1} \\ \nonumber &\hspace{10mm} - (1 - f_{i,\bar{\alpha}}(x) )\bar{z}_i(k) - (1 - f_{i,\underline{\alpha}}(x) )\underline{z}_i(k), \\
    &f_{i,\Delta}(x) - f_{i,\bar{\alpha}}(x) \cdot \Big(\bar{z}_i - \sum_{j=1}^n\hat{A}_{ij} f_{j,z}(x) \Big) \label{eq:poly_offline_p-2} \\ \nonumber & \hspace{10mm}- f_{i,\underline{\alpha}}(x) \cdot \Big(\sum_{j=1}^n\hat{A}_{ij} f_{j,z}(x) - \underline{z}_i\Big),  \\
    &f_{i,\bar{\alpha}}(x)\Big(1-f_{i,\bar{\alpha}}(x)\Big), \label{eq:poly_offline_p-3} \\
    &f_{i,\underline{\alpha}}(x)\Big(1-f_{i,\underline{\alpha}}(x)\Big). \label{eq:poly_offline_p-4}
\end{align}
\end{subequations}
\endgroup
\begingroup
\allowdisplaybreaks
\begin{subequations}\label{eq:composite_poly}
    \begin{align}
        &q_{i,1}(x):= \frac{\eqref{eq:poly_offline_p-1}}{\Pi_{j=0}^{N-1}(x - g^j)}, \quad \forall i \in 1,\cdots,n, \label{eq:composite_poly-1} \\
        &q_{i,2}(x):= \frac{\eqref{eq:poly_offline_p-2}}{\Pi_{j=0}^{N-1}(x - g^j)}, \quad \forall i \in 1,\cdots,n, \label{eq:composite_poly-2} \\
        &q_{i,3}(x):= \frac{\eqref{eq:poly_offline_p-3}}{\Pi_{j=0}^{N-1}(x - g^j)}, \quad \forall i \in 1,\cdots,n, \label{eq:composite_poly-3} \\
        &q_{i,4}(x):= \frac{\eqref{eq:poly_offline_p-4}}{\Pi_{j=0}^{N-1}(x - g^j)}, \quad \forall i \in 1,\cdots,n. \label{eq:composite_poly-4} 
    \end{align}
\end{subequations}
\endgroup
The maximal possible degrees of each of the composition polynomials, $q_{i,l}$, $i \in\{1,\cdots,n\}$ and $l \in\{1,\cdots,4\}$, are given as:
\begingroup
\allowdisplaybreaks
\begin{subequations}\label{eq:composite_poly_deg}
    \begin{align}
        &\text{deg}(q_{i,1}) \leq \max\{D_{i,z} + D_{i,\bar{\alpha}} +  D_{i,\underline{\alpha}} - N, 0 \}, \label{eq:composite_poly_deg-1} \\ 
        &\text{deg}(q_{i,2}) \leq \max\{D_{i,z}+\max\{D_{i,\bar{\alpha}},D_{i,\underline{\alpha}}\} - N, 0\}, \label{eq:composite_poly_deg-2} \\
        &\text{deg}(q_{i,3}) \leq \max\{2D_{i,\bar{\alpha}} - N,0\}, \label{eq:composite_poly_deg-3} \\
        &\text{deg}(q_{i,4}) \leq \max\{2D_{i,\underline{\alpha}}  -N, 0\}. \label{eq:composite_poly_deg-4} 
    \end{align}
\end{subequations}
\endgroup
The maximal possible degrees are straight derived from the degree of the individual polynomials $f_{i,z}$, $f_{i,\bar{\alpha}}$, $ f_{i,\underline{\alpha}}$, and $f_{i,\Delta}$ ($D_{i,z}$, $D_{i,\bar{\alpha}}$, $ D_{i,\underline{\alpha}}$ and $D_{i,\Delta}$). Note that each of \eqref{eq:composite_poly-1}-\eqref{eq:composite_poly-4} is a polynomial only if the denominator divide the numerator, which is the case only if all the numbers of the execution trace comply with~\eqref{eq:control_slack-1}-\eqref{eq:control_slack-2} and~\eqref{eq:equi_alpha-1}. This implies that to verify that~\eqref{eq:poly-2} hold, the verifier only needs to check whether there exists a polynomial with at most $\text{deg}(q_{i,l})$ that matches $q_{i,l}$ for all $x\in\Fc^*_q\setminus\{g^0,\cdots,g^{N-1}\}$, $\forall i \in\{1,\cdots,n\}$, and $\forall l \in\{1,\cdots,4\}$. In fact, there is a low complexity degree testing method that leverages the fast Reed-Solomon interactive oracle proof of Proximity (FRI) to verify whether the upper bound of the degree holds true. If the bound does not hold, then the execution trace fails the verification test. Because the low-complexity degree testing method based on FRI is the main reason the whole reformulation from~\eqref{eq:control} to polynomial constraints~\eqref{eq:poly} and composition polynomials~\eqref{eq:composite_poly}, we will go through the rough idea of low-degree testing even though more detailed explanations are available in~\cite{ZKSTARK_blog4,ben2018fast}.

\subsection{Low-degree testing}
The verifier can test the degree of all the composition polynomials in~\eqref{eq:composite_poly} individually; however, to further simplify the verification process, those composition polynomials are usually combined in a single one for the verification in zk-STARK. The verifier randomly chooses $\gamma_{i,l}\in \Fc_q^*$ for all $i=\{1,\cdots,n\}$ and $l\in\{1,\cdots,4\}$ and requests that the prover constructs the composition polynomial in the following:
\begin{align}\label{eq:composite_poly_tot}
    &Q(x) = \sum_{i = 1}^n\sum_{l = 1}^4 \gamma_{i,l} \cdot q_{i,l}(x), \\ \nonumber &\hspace{5mm}\text{deg}(Q):= D_Q \leq \bar{D}_Q = \max_{i\leq n}(D_{i,z} + D_{i,\bar{\alpha}} +  D_{i,\underline{\alpha}}) - N.
\end{align}
If the composition polynomial has the degree bounded by $\bar{D}_Q$ as defined in~\eqref{eq:composite_poly_tot}, then it is highly probable that the execution trace is valid because if the execution trace is invalid, then~\eqref{eq:composite_poly_tot} is no longer a polynomial because the denominator does not divide the numerator for at least one composition polynomials in~\eqref{eq:composite_poly}. 

The idea of FRI is halving the degree of the target polynomial $Q(x)$, for each query from the verifier to the prover. Repeating the steps for $\lfloor \log_2 \bar{D}_Q \rfloor +1$ times leads to a constant polynomial, which can be translated to exponentially fewer numbers of queries compared to $\bar{D}_Q$. The FRI protocol engages the prover to achieve exponentially fewer queries. First, the prover writes the composition polynomial as:
\begin{align*}
    Q(x) = Q_e(x^2) + x Q_o(x^2),
\end{align*}
where $Q_e$ and $Q_o$ are constructed, respectively, by the coefficients of even and odd powers of $x$ in $Q$. After receiving a $\beta_0\in\Fc_q^*$ randomly picked by the verifier, the prover commits a new polynomial, shown in the following:
\begin{align}\label{eq:halve_deg}
    Q_1(x) = Q_e(x) + \beta_0 Q_o(x).
\end{align}
The new polynomial $Q_1$ has a degree no bigger than $\frac{\bar{D}_Q}{2}$. The verifier repeats this type of query for $\lfloor \log_2 \bar{D}_Q \rfloor +1$ times and checks whether the degree is $0$ (a constant polynomial). If not, then the execution trace generated from the prover fails the low-degree testing.  The procedure described above is called commit phase in the context of FRI such that the prover commits to a sequence of polynomials, $Q$, $Q_1$,$\cdots$,$Q_{\lfloor \log_2 \bar{D}_Q \rfloor + 1}$, with degree halving through each step by constant $\beta_0,\cdots,\beta_{\lfloor \log_2 \bar{D}_Q \rfloor}$ sent by the verifier. There is an additional phase, known as the query phase in FRI, such that the verifier can ensure that the prover did not cheat in the commit phase. The idea is that the verifier samples a random $x_s\in\Fc^*_q\setminus\{g^0,\cdots,g^{N-1}\}$ and queries the values of $Q(x_s)$, $Q(-x_s)$ and $Q_1(x_s^2)$. The verifier knows the following equations should hold if the prover did not cheat:
\begin{subequations}\label{eq:query_phase}
\begin{align}
    Q(x_s) &=  Q_e(x_s^2) + x_s Q_o(x_s^2), \label{eq:query_phase-1}\\
    Q(-x_s) &= Q_e(x_s^2) - x_s Q_o(x_s^2), \label{eq:query_phase-2}\\
    Q_1(x_s^2) &= Q_e(x_s^2) + \beta_0 Q_o(x_s^2). \label{eq:query_phase-3}
\end{align}
\end{subequations}
Using~\eqref{eq:query_phase}, the verifier can compute $Q_1(x_s^2)$ with the queried $Q(x_s)$ and $Q(-x_s)$. The computed $Q_1(x_s^2)$ should match the queried $Q_1(x_s^2)$. The queries go further down to the lowest-degree polynomial $Q_{\lfloor \log_2 \bar{D}_Q \rfloor + 1}$. With multiple $x_s$ sample points, the verifier can conclude that with high probability the prover did not cheat in the commit phase.

\section{Numerical Studies}\label{sec:num}
We use a very simple test case to validate the reformulation and verification methods explained in the last section because our goal here is only a proof of concept. Instead of using a large finite field,  $\Fc_{2^M}$ ($M\geq 32$), to accurately capture the real state variables $z$ and $\Delta$ in~\eqref{eq:control_slack}, we assume that all the variables in~\eqref{eq:control_slack} only take values in a field, $\Fc_{331}$. The specific initialization of the dynamical system is shown in the following:
\begin{align*}
    \hat{A} = \begin{bmatrix} 1 & 0 \\ -1 & 1\end{bmatrix}, \quad \bar{z} = \begin{bmatrix}100 \\ 100\end{bmatrix}, \quad \underline{z} = \begin{bmatrix}0 \\ 40\end{bmatrix}, \quad z(0) = \begin{bmatrix}3 \\ 100\end{bmatrix}.
\end{align*}
We consider a scenario that the prover fully follows~\eqref{eq:control_slack}, and then we show how it passes through all the checkpoints of the verification process. First, the prover successfully generates the execution trace with length $N+1$ in the online stage because both~\eqref{eq:equi_alpha-2} and~\eqref{eq:equi_alpha-3} are satisfied for every iteration. The prover next generates and commits the polynomials such that~\eqref{eq:poly_variables} holds, which requires defining a cyclic subgroup of $\Fc_{331}^*$. Because $N=29$ in the simulation setup, the size of the cyclic subgroup, $\Sc$, is $N+1$, i.e., $|\Sc| = N+1 = 30$. Such a cyclic subgroup is well defined with the generator $g=2$ and $\Sc = \{g^0,g^1,\cdots, g^N\}$ $=\{1,2,4,8,16,32,64,128,256,181,31,62,124,248,165,$ $ 330,329,327,323,315,299,267,203,75,150,300,269,207,$ $83,166\}$. 
The prover finds that many polynomials are actually constant polynomials (or have degree 0), as shown in Table~\ref{tab:degree}. The reason is because every variable stays constant except $z_2$, $\underline{\alpha}_2$, and $\Delta_2$. With the lower bound of $z_2$ becoming active from iteration 20, the evolution of variables $z_2$, $\underline{\alpha}_2$ ,and $\Delta_2$ is nontrivial and the prover finds the degrees close to the upper bound.
\begin{table}[h]
    \centering
    \begin{tabular}{|c|c|c|c|c|c|c|c|c|}\hline
        &$f_{1,z}$ & $f_{2,z}$ & $f_{1,\Delta}$ & $f_{2,\Delta}$ & 
        $f_{1,\underline{\alpha}}$ & $f_{2,\underline{\alpha}}$ & 
        $f_{1,\bar{\alpha}}$ & $f_{2,\bar{\alpha}}$
        \\ \hline
        Degree & 0 & 29 & 0 & 28 & 0 & 28 & 0 & 0 \\ \hline
    \end{tabular}
    \caption{The degrees of the polynomials.}
    \label{tab:degree}
\end{table}
With those polynomials defined, the prover then commits the composition polynomials with the degree shown in Table~\ref{tab:degree_2}. Here is a checkpoint for the verifier: The verifier needs to ensure that all the $q-$composition polynomials are actually generates by~\eqref{eq:composite_poly} instead of grabbing some random low-degree polynomials to pass the low-degree testing in the next step. The verifier queries the outputs of the $f$-polynomials evaluated at some random sample points in $\Fc^*_{331}$. Using those outputs, the verifier can easily compute the outputs of all the $q$-composition polynomials evaluated at those random sample points. The $q$-composition polynomials committed by the prover should give the same outputs as those computed by the verifier to pass the test. The details why computing the denominators of~\eqref{eq:composite_poly} does not burden the verifier are in~\cite{ZKSTARK_blog3}.
\begin{table}[h]
    \centering
    \begin{tabular}{|c|c|c|c|c|c|c|c|c|}\hline
         & $q_{1,1}$ & $q_{2,1}$ & $q_{1,2}$ & $q_{2,2}$ & $q_{1,3}$ & $q_{2,3}$ & $q_{1,4}$ &$q_{2,4}$\\ \hline
        Degree & 0 & 28 & 0 & 28 & 0 & 27 & 0 & 0 \\ \hline
    \end{tabular}
    \caption{The degrees of the composition polynomials.}
    \label{tab:degree_2}
\end{table}
\begin{table}[h]
    \centering
    \begin{tabular}{|c|c|c|c|c|c|c|c|}\hline
         $\gamma_{1,1}$ & $\gamma_{2,1}$ & $\gamma_{1,2}$ & $\gamma_{2,2}$ & $\gamma_{1,3}$ & $\gamma_{2,3}$ & $\gamma_{1,4}$ &$\gamma_{2,4}$\\ \hline
        261 & 308 & 225 & 47 & 236 & 41 & 43 &  212\\ \hline
    \end{tabular}
    \caption{The $\gamma$ values randomly sampled by the verifier.}
    \label{tab:gamma}
\end{table}
\begin{table*}[h]
    \centering
    \begin{tabular}{|c|c|c|c|c|c|c|} \hline
        $x_s$ &  & $Q_1(x_s^2)$ & $Q_2(x_s^4)$ & $Q_3(x_s^8)$ & $Q_4(x_s^{16})$ & $Q_5(x_s^{32})$ \\ \hline
        \multirow{2}{*}{87} & Computed value from~\eqref{eq:query_phase} & 128   & 22 & 101 & 39 & 229 \\ \cline{2-7}
         & Queried value & 128   & 22 & 101 & 39 & 229 \\ \hline
        \multirow{2}{*}{291} & Computed value from~\eqref{eq:query_phase} & 324 & 35 & 219 & 110 & 229 \\ \cline{2-7}
         & Queried value & 324 & 35 & 219 & 110 & 229 \\ \hline
    \end{tabular}
    \caption{The results of the query phase. The verifier makes two sample points: $x_s=87$ and $x_s = 291$. The computed values from~\eqref{eq:query_phase} match the queried values throughout all $Q_i$, $i=1,\cdots,5$. }
    \label{tab:query_phase}
\end{table*}

The last stage is the low-degree testing. The verifier samples some random weights for each $q$-composition polynomial as shown in Table~\ref{tab:gamma}. The coefficients of $x^i$ for $i=0,\cdots,28$ for the resulting combined composition polynomial, $Q$, are listed in the following:
\begin{align*}
&Q: 72, 260, 273, 61, 25, 37, 225, 18, 311, 255, 292, 157, 83, 151,  \\ 
&\; 31,  172, 203, 244, 39, 65, 136, 317, 91, 84, 238, 325, 94, 24, 275 
\end{align*}
The prover next commits the sequence of the polynomials halving the degree with~\eqref{eq:halve_deg} and randomly sampled $\beta_0,\cdots,\beta_4$ from the verifier. The coefficients of $x^i$ ($i$ from $0$ to the highest power of $x$) of those polynomials are shown in the following: 
\begin{align*}
Q_1 &: 85, 94, 242, 259, 241, 184, 74, 172, 149, 125, 36, 29, 6, 29,  \\ &\quad 275, (\text{ deg}(Q_1) = 14, \quad \beta_0 = 149)   \\
Q_2 &: 18, 75, 300, 50, 324, 209, 179, 275, \; (\text{deg}(Q_2) = 7, \; \beta_1 = 200) \\
Q_3 &: 88, 126, 166, 215, \quad (\text{deg}(Q_3) = 3, \quad \beta_2 = 23) \\
Q_4 &: 204, 117, \quad (\text{deg}(Q_4) = 1, \quad \beta_3 = 106)  \\
Q_5 &: 229, \quad (\text{deg}(Q_5) = 0, \quad \beta_4 = 252).
\end{align*}
The verifier knows that if $Q$ has the degree bounded by 28 (according to Table~\ref{tab:gamma}), the degree of the combined composition polynomial should become zero after five halving actions. In this numerical example, the prover has the last composition polynomial, $Q_5 = 229$, as a constant polynomial. Because the prover does not cheat in our simulation, it passes the tests in the query phase,  as shown in Table~\ref{tab:query_phase}. No matter which sample point, $x_s$, the verifier selected, all the outputs of $Q_5$ are the same constant. This implies that $Q_5$ is indeed a constant polynomial, so the original combined composition polynomial, $Q$, passes the low degree testing. 


\section{Conclusion}
In this paper, we proposed a two-stage verification framework that leverages zk-STARK for a more secure and reliable implementation of projected gradient controllers. The controller/prover takes additional offline computations to generate a proof of accurate computation - computational integrity - to convince the verifier that the designated projected gradient control dynamics are accurately executed for generating the sequence of control. The proof-of-concept approach presented here can be leveraged by power grid control systems such as Advanced Distribution Management Systems (ADMS) and Distributed Energy Resource Management Systems (DERMS) to implement an inherently secure power grid controls. Our future work will include more engaging numerical studies with a large finite field to further validate the capability of the proposed framework, along with a lab based assessment and demonstration of the proposed approach. We will also look into other more general controller dynamics and improve the two-stage verification framework to accommodate those scenarios.


\bibliographystyle{IEEEtran}
\bibliography{acmart}

\end{document}